\documentclass[12pt]{amsart}
\usepackage{amssymb}
\usepackage[all]{xy}
\newtheorem{theorem}{Theorem}[section]
\newtheorem{lemma}[theorem]{Lemma}

\def\to{\rightarrow}

\def\m{\mathbb}
\def\c{\mathcal}
\def\b{\mathbf}
\def\sl{\langle}
\def\sr{\rangle}

\begin{document}
\title[Stability of additive functional equation]{On the stability of a generalized\\
additive functional equation$^\star$}
\footnotetext{$^\star$The final publication is available at Springer via http://dx.doi.org/10.1007/s00010-015-0370-2}
\author[M. M. Sadr]{Maysam Maysami Sadr}
\address{Department of Mathematics\\
Institute for Advanced Studies in Basic Sciences\\
P.O. Box 45195-1159, Zanjan 45137-66731, Iran}
\email{sadr@iasbs.ac.ir}
\keywords{Additive functional equation; Stability; Topological group; Amenability}
\begin{abstract}
We consider Hyers-Ulam stability of a functional equation for continuous functions
on a space on which a topological group acts, analogous to the additive functional equation on a group.
We show, among other things, that our generalized additive equation, for continuous functions on a homogenous space
of a strongly amenable topological group, is stable provided that the canonical projection from that group to its 
homogenous space is a fiber bundle.
\end{abstract}
\maketitle
\section{Introduction}
The story of stability of functional identities was began by Ulam \cite{Ulam1} when he proposed
the following question about a special type of stability of homomorphisms between groups.
\begin{quote}
Let $G_1$ and $G_2$ be groups and suppose that $G_2$ has a metric $d$. Given $\epsilon>0$ does there exist $\delta>0$ such that
if a map $f:G_1\to G_2$ satisfies $d(f(yz),f(y)f(z))<\delta$ then there is a group homomorphism $F:G_1\to G_2$ satisfying
$d(F(y),f(y))<\epsilon$?
\end{quote}
In 1941, Hyers \cite{Hyers1} answered affirmatively this question for additive mappings between Banach spaces. The Ulam stability problem and
its generalizations, not only for the equation of homomorphism but also other types of functional identities, has been considered
and developed by many mathematicians. For the history of developments see \cite{Jung1} and \cite{HyersRassias1}.

The method of \emph{amenability} and \emph{invariant means} for the study of stability of functional identities, at the first time, was used by
Sz\'{e}kelyhidi \cite{Szekelyhidi2}. Then Forti \cite{Forti1} showed that for any amenable group $G$ and any map $h$ from $G$ to a Banach space
$E$ satisfying $|h(yz)-h(y)-h(z)|<\delta$ there exists an additive mapping $H:G\to E$ satisfying $|H(y)-h(y)|<\delta$.
The easy proof of Forti has been applied and extended by many authors, for instance see \cite{Szekelyhidi1}, \cite{Forti2},
\cite{ElqorachiManarRassias1}, \cite{Zlatos1}, \cite{Yang2}. See also the survey paper \cite{HyersRassias1}.

Sz\'{e}kelyhidi \cite{Szekelyhidi1} have considered the stability properties of a generalized additive equation
$f(xy)=f(x)\alpha(y)+h(y)$ for a function $f$ on a $G$-set $X$ and functions $h,\alpha$ on $G$ with values in a linear space $E$ where $G$
is a group. In this note we consider and prove some generalizations of the results of \cite{Szekelyhidi1}, as described below.

In Section 3 we consider the stability of the equation
$$F(xy)=F(x)\cdot y+H(y),$$
where $H$ is a map on a discrete group $G$ with values in a Banach
right $G$-module $E$, and $F$ is a map on a right $G$-set $X$ with values in $E$. Indeed, we replace the complex or real valued map $\alpha$ in
\cite{Szekelyhidi1} by a fixed action of $G$ on the Banach space $E$.
It is shown in Theorem \ref{T1} that the mentioned equation has Hyers-Ulam stability if $G$ is an amenable group and $E$ is a dual Banach $G$-module.
Moreover, in this case, it is shown that $H$ is an additive mapping.

Almost recently there has been initiated many researches and interest about amenability notions of non-locally compact topological groups,
see \cite{GrigorchukHarpe1} and the references therein. So it would be interesting to consider and generalize the known results of stability
of different types of equations on discrete amenable groups to topological groups with amenable properties. Indeed, we do that in Section 4.
In Theorems \ref{T2} and \ref{T3} we consider the stability of the same functional equation of Theorem \ref{T1} but for topological groups
and $G$-spaces, and with the trivial action of $G$ on $E$.
Rather than Banach $G$-modules as in Theorem \ref{T1}, we restrict our selves in Theorems \ref{T2} and \ref{T3} to the case in which
$G$ acts trivially on $E$. The reason is that with the assumption of nontrivial action
the conditions under which Theorems \ref{T2} and \ref{T3} are satisfied (or at least under which the author was able to prove them)
become so strong that imply $F=0$ and $H=0$, the case with no interest. In Theorem \ref{T2} we show that the mentioned equation has
Hyers-Ulam stability for continuous functions  whenever $G$ is a \emph{strongly} amenable topological group and $X$ is a homogenous
$G$-space for which the canonical projection from $G$ to $X$ is a fiber bundle, and $E$ is a dual Banach space.
Note that homogenous spaces are one of the most important classes of spaces on which topological groups act.
Theorem \ref{T3} shows that the stability is satisfied for uniformly continuous functions (see section 2 for exact definitions)
whenever $G$ is an amenable topological group and $X$ is an arbitrary $G$-space.
\section{Preliminaries}
For any topological space $X$ we denote by $\b{B}(X)$ and $\b{C}(X)$ the space of real valued bounded and continuous
functions on $X$, respectively. We always consider $\b{B}(X)$ as a Banach space with uniform norm. Also we let $\b{CB}(X)=\b{C}(X)\cap\b{B}(X)$.

Let $G$ be a topological group and $X$ be a right $G$-space. (This means that there is a jointly continuous map $(x,y)\mapsto xy$ from
$X\times G\to X$ satisfying $xe=x$ and $x(yz)=(xy)z$, for $x\in X$ and $y,z\in G$.)
Then $\b{RU}(X)$ denotes the space of real functions $f$ on $X$ for which
there is an open neighborhood $U_\epsilon$ in $G$ of $e$, for every $\epsilon>0$, such that $\sup_{x\in X}|f(xy)-f(x)|<\epsilon$
for every $y\in U_\epsilon$. Or equivalently, for every $y\in G$ and $\epsilon>0$ there is an open $W\subseteq G$ containing $y$ such that
$\sup_{x\in X}|f(xy)-f(xy')|<\epsilon$ for every $y'\in W$.
Such functions are called \emph{right uniform}. The space of left uniform functions, denoted $\b{LU}(X)$,
contains functions $f$ such that for every $\epsilon>0$ and every $x\in X$ there is an open $V$ in $X$ with $x\in V$ and
$\sup_{y\in G}|f(xy)-f(x'y)|<\epsilon$ for every $x'\in V$. We let $\b{U}(X)=\b{RU}(X)\cap\b{LU}(X)$. (Note that
we always have $\b{LU}(X)\subseteq\b{C}(X)$, but in general $\b{RU}(X)\varsubsetneqq\b{C}(X)$. For example consider the action of $G=\b{T}^1$,
the unite circle group, on $X=\m{C}$ from right by rotation around the origin. Then $f:\m{C}\to\m{R}$ defined by $f(0)=0$ and $f(x)=1$ for
$x\in\m{C}-0$ is right uniform. Also note that if $G$ is discrete then $\b{U}(X)=\b{LU}(X)$.) Other function spaces, e.g.
$\b{RUCB}(X)$, are defined in the obvious way. Also the analogue of these definitions are satisfied for a left $G$-space.
(Note that if $G_l$ and $G_r$ denote $G$ as a left $G$-space and right $G$-space, respectively, then
$\b{RU}(G_l)=\b{RU}(G_r)$ and $\b{LU}(G_l)=\b{LU}(G_r)$.)

Let $G$ and $X$ be as above. A subspace $A$ of $\b{B}(X)$ containing constant functions
is called \emph{right invariant} if for every $f\in A$  and $y\in G$ the function $\c{R}_yf$, defined by $(\c{R}_yf)(x)=f(x.y)$, belongs to $A$.
(The notation $\c{L}_yf$ of left translation by $y$ of a function $f$ on a left $G$-space is defined similarly.)
Then a bounded linear functional $m$ on $A$ with $m(1)=\|m\|=1$ is called a \emph{right invariant mean} if $m(\c{R}_yf)=m(f)$, for every
$f\in A$ and $y\in G$. Left invariant subspaces and left invariant means for left $G$-spaces are defined similarly.
(We remark that there is no common use of adjectives "right" and "left" in the literatures.)

A topological group $G$ is called amenable (resp. strongly amenable) (\cite{GrigorchukHarpe1})
if there is a left invariant mean on $\b{LUCB}(G)$ (resp. $\b{CB}(G)$). Because the inverse map exists and is continuous on a topological
group, amenability (resp. strong amenability) is equivalent to existence of a right invariant mean on $\b{RUCB}(G)$ (resp. $\b{CB}(G)$).
It is well known that for a locally compact group amenability and strong amenability coincide (\cite{Pier1}).
We refer the reader to \cite{GrigorchukHarpe1} and references therein for the
notion of (strong) amenability of non-locally compact groups and many interesting examples of such groups.
(We remark that for a right $G$-space $X$ if there is a right invariant mean on $\b{RUCB}(X)$ then in the terminology of \cite{Greenleaf1}
it is said that $G$ \emph{has amenable action} on $X$. Also see \cite[Section 24.B]{Pier1} for different notions of an amenable action.)
\begin{lemma}\label{L1}
Let $G$ be a topological group and $X$ be a right $G$-space. Let $m$ be a right invariant mean on $\b{RUCB}(G)$ (resp. $\b{CB}(G)$)
and fix an element $x_0$ in $X$. Then the real function $m'$ on $\b{RUCB}(X)$ (resp. $\b{CB}(X)$), defined by $m'(f)=m(y\mapsto f(x_0y))$,
is a right invariant mean.
\end{lemma}
\begin{proof}
Straightforward.
\end{proof}
Let $G$ and $X$ be as above. For any Banach space $E$ we can define the various spaces of functions on $X$ with
values in $E$ instead $\m{R}$, e.g. $\b{RUCB}(X;E)$ is the space of all bounded continuous right uniform functions with values in $E$.

Let $G$ be a topological group. Then by a Banach left $G$-module we mean a Banach space $E$ together with a continuous map $(y,\xi)\mapsto y\cdot\xi$
from $G\times E$ to $E$ satisfying $y\cdot (y'\cdot\xi)=(yy')\cdot\xi$ and $e\cdot\xi=\xi$, and such that the map $\xi\mapsto y\cdot\xi$ is an
isometric linear automorphism on $E$. The Banach right $G$-modules are defined similarly. If $E$ is a Banach left $G$-module then $E^*$
becomes a right Banach $G$-module with the action defined by $\sl\alpha\cdot y,\xi\sr=\sl\alpha,y\cdot\xi\sr$ for every $\alpha\in E^*$.
\begin{lemma}\label{L2}
Let $G$ be a topological group, $X$ be a right $G$-space, and $E$ be a left Banach $G$-module. Suppose that $m$ is a
right invariant mean on $\b{RUCB}(X)$ (resp. $\b{CB}(X)$). Then $\tilde{m}$, defined by $\tilde{m}(f)(\xi)=m(x\mapsto\sl f(x),\xi\sr)$ ($\xi\in E$),
is a well-defined bounded linear operator from $\b{RUCB}(X;E^*)$ (resp. $\b{CB}(X;E^*)$) to $E^*$ with the following properties.
\begin{enumerate}
\item[i)] $\|\tilde{m}\|=1$, and for every $c$, $\tilde{m}(c)=c$, where $c$ mutually
denotes an element of $E^*$ and the constant function with value $c$.
\item[ii)] $\tilde{m}(\c{R}_yf)=\tilde{m}(f)$ for every $y\in G$ and every $f$.
\item[iii)] $\tilde{m}(x\mapsto f(x)\cdot y)=\tilde{m}(f)\cdot y$ for every $y\in G$ and every $f$.
\end{enumerate}
\end{lemma}
\begin{proof}
It is easily checked that if $f$ is in $\b{RUCB}(X;E^*)$ (resp. $\b{CB}(X;E^*)$) then for every $\xi\in E$ the map $x\mapsto\sl f(x),\xi\sr$
is in $\b{RUCB}(X)$ (resp. $\b{CB}(X)$) and the assignment $\xi\mapsto m(x\mapsto\sl f(x),\xi\sr)$ defines a bonded linear functional
on $E$. So $\tilde{m}$ is a well-defined linear map. We have,
\begin{align*}
\|\tilde{m}\|&=\sup_{\|f\|_\infty\leq1}\|\tilde{m}(f)\|\\
&=\sup_{\|f\|_\infty\leq1}\sup_{\|\xi\|\leq1}|m(x\mapsto\sl f(x),\xi\sr)|\\
&\leq\sup_{\|f\|_\infty\leq1}\sup_{\|\xi\|\leq1}\|x\mapsto\sl f(x),\xi\sr\|_\infty\\
&=\sup_{\|f\|_\infty\leq1}\sup_{\|\xi\|\leq1}\sup_{x\in X}|\sl f(x),\xi\sr|\\
&\leq\sup_{\|f\|_\infty\leq1}\sup_{\|\xi\|\leq1}\sup_{x\in X}\|f(x)\|\|\xi\|\leq1.
\end{align*}
Also for every $c\in E^*$, $\tilde{m}(c)(\xi)=m(x\mapsto\sl c,\xi\sr)=\sl c,\xi\sr$. So $\tilde{m}(c)=c$. This completes the proof of i).
ii) is proved by,
\begin{align*}
\tilde{m}(\c{R}_yf)(\xi)&=m(x\mapsto\sl \c{R}_yf(x),\xi\sr)\\
&=m(x\mapsto\sl f(xy),\xi\sr)\\
&=m(x\mapsto\sl f(x),\xi\sr)=\tilde{m}(f).
\end{align*}
iii) is proved by,
\begin{align*}
\tilde{m}(x\mapsto f(x)\cdot y)(\xi)&=m(x\mapsto\sl f(x)\cdot y,\xi\sr)\\
&=m(x\mapsto\sl f(x),y\cdot\xi\sr)\\
&=\tilde{m}(f)(y\cdot\xi)=(\tilde{m}(f)\cdot y)(\xi).
\end{align*}
\end{proof}
\section{The result in discrete case}
\begin{theorem}\label{T1}
Let $G$ be an amenable (discrete) group, $X$ be a right $G$-set and $E$ be a Banach left $G$-module. Suppose there are given maps $f:X\to E^*$
and $h:G\to E^*$ satisfying
$$\|f(xy)-f(x)\cdot y-h(y)\|\leq\delta,$$
for every $x\in X$ and $y\in G$. Then there exist maps $F:X\to E^*$ and $H:G\to E^*$ satisfying,
$$H(yz)=H(y)\cdot z+H(z),\hspace{10mm}F(xy)=F(x)\cdot y+ H(y),$$
$$\|H(y)-h(y)\|\leq\delta,\hspace{10mm}\|F(x)-f(x)\|\leq2\delta.$$
\end{theorem}
\begin{proof}
It follows from Lemmas \ref{L1} and \ref{L2} that there exists a bounded linear map $m:\b{B}(X;E^*)\to E^*$ satisfying properties analogous to those
of $\tilde{m}$ in Lemma \ref{L2}. Since
$\|f(xy)-f(x)\cdot y\|\leq\|h(y)\|+\delta$, we may define a map $H:G\to E^*$ by $H(y)=m(x\mapsto f(xy)-f(x)\cdot y)$.
We have
\begin{align*}
&H(y)\cdot z+H(z)\\
=&m(x\mapsto f(xy)-f(x)\cdot y)\cdot z+m(x\mapsto f(xz)-f(x)\cdot z)\\
=&m(x\mapsto f(xy)\cdot z-f(x)\cdot (yz))+m(x\mapsto f((xy)z)-f(xy)\cdot z )\\
=&m(x\mapsto f(x(yz))-f(x)\cdot(yz))=H(yz).
\end{align*}
The norm inequality is proved by,
\begin{align*}
\|H(y)-h(y)\|&=\|m(x\mapsto f(xy)-f(x)\cdot y)-h(y)\|\\
&=\|m(x\mapsto f(xy)-f(x)\cdot y-h(y))\|\\
&\leq\|m\|\|x\mapsto f(xy)-f(x)\cdot y-h(y)\|_\infty\leq\delta.
\end{align*}
It follows from the assumptions and the above inequality that
$$\|f(xz)-f(x)\cdot z-H(z)\|\leq2\delta.$$
So,
\begin{align*}
\|f(xz)\cdot z^{-1}-H(z)\cdot z^{-1}\|&=\|f(xz)-H(z)\|\\
&\leq\|f(x)\cdot z\|+2\delta=\|f(x)\|+2\delta.
\end{align*}
On the other hand a result analogous to Lemma \ref{L2} shows that there exists a bounded linear map $n:\b{B}(G;E^*)\to E^*$ satisfying
$\|n\|=1$, $n(c)=c$ ($c\in E^*$), $n(\c{L}_yg)=n(g)$, and $n(z\mapsto g(z)\cdot y)=n(g)\cdot y$.
Therefore we can define a map $F:X\to E^*$ by
$$F(x)=n(z\mapsto f(xz)\cdot z^{-1}-H(z)\cdot z^{-1})$$
and then we have,
\begin{align*}
F(x)\cdot y+H(y)&=n(z\mapsto f(xz)\cdot z^{-1}-H(z)\cdot z^{-1})\cdot y+ H(y)\\
&=n(z\mapsto f(xz)\cdot (z^{-1}y)-H(z)\cdot (z^{-1}y)+ H(y))\\
&=n(z\mapsto f(xz)\cdot (z^{-1}y)+H(z^{-1}y))\\
&=n(z\mapsto f(x(yz))\cdot z^{-1}+H(z^{-1}))\\
&=n(z\mapsto f((xy)z)\cdot z^{-1}-H(z)\cdot z^{-1})=F(xy).
\end{align*}
The norm inequality is proved by,
\begin{align*}
\|F(x)-f(x)\|&=\|n(z\mapsto f(xz)\cdot z^{-1}-H(z)\cdot z^{-1})-f(x)\|\\
&=\|n(z\mapsto f(xz)\cdot z^{-1}-H(z)\cdot z^{-1}-f(x))\|\\
&\leq\|n\|\|z\mapsto f(xz)\cdot z^{-1}-H(z)\cdot z^{-1}-f(x)\|_\infty\\
&=\sup_{z\in G}\|f(xz)\cdot z^{-1}-H(z)\cdot z^{-1}-f(x)\|\\
&=\sup_{z\in G}\|f(xz)-H(z)-f(x)\cdot z\|\leq2\delta.
\end{align*}
\end{proof}
As the following observation shows, the map $H$ in the above theorem is unique. Let $H':G\to E^*$ be another additive map
satisfying $\|H'(y)-h(y)\|\leq\delta$. Then $H-H':G\to E^*$ is additive and bounded, and so $H-H'=0$. In general the map $F$
is not unique. For example consider the following case. Let $G=(\m{R},+)$ and suppose that $G$ acts on $X=\m{R}$ by translation
and $G$ acts on the Banach space $E=E^*=\m{R}$ trivially. Let $h=f=\mathrm{id}$. Then $H=\mathrm{id}$, $F=\mathrm{id}$,
and $F=\mathrm{id}+\delta/2$ are satisfied in the theorem.
\section{The results in the continuous case}
We shall need the following lemma.
\begin{lemma}\label{L3}
Let $G$ be a topological group and $E$ be a Banach space. Suppose that $H:G\to E$ is an additive map and $h:G\to E$ is a continuous map
such that for a constant $\delta>0$, $\|H(y)-h(y)\|\leq\delta$ for every $y\in G$. Then $H$ is also continuous.
\end{lemma}
\begin{proof}
It is enough to show that $H$ is continuous at $e$. Suppose, on the contrary, that $H$ is not continuous at $e$. Then there are $\epsilon>0$
and a net $(y_\lambda)_\lambda$ in $G$ such that $y_\lambda\to e$ and $\|y_\lambda\|>\epsilon$. Let $k\in\mathbb{N}$ be such that
$k\epsilon>3\delta$. Then the net $(y_\lambda^k)_\lambda$ converges to $e$ and $\|H(y_\lambda^k)\|>k\epsilon>3\delta$. On the
other hand, for some $\lambda_0$, $\|h(y_{\lambda_0}^k)-h(e)\|<\delta$. So we get the following contradiction.
$$\|H(y_{\lambda_0}^k)\|\leq\|h(e)\|+\|h(y_{\lambda_0}^k)-h(e)\|+\|H(y_{\lambda_0}^k)-h(y_{\lambda_0}^k)\|\leq3\delta.$$
\end{proof}
Let $G$ be a topological group and $L$ be a closed subgroup of $G$. As usual, $G/L$ denotes the space of right cosets of $L$ in $G$
together with the quotient topology, i.e. the topology coinduced by the canonical projection $\pi:G\to G/L$ defined by $\pi(y)=Ly$.
Then $G/L$ is also a right $G$-space by the action $(Ly,y')\mapsto Lyy'$. By a \emph{local cross section} around a point $Ly$ of $G/L$
we mean a continuous map $s$ from an open neighborhood $U$ of $Ly$ in $G/L$ to $G$ such that $\pi s=\mathrm{id}_U$. It is well known
that the map $\pi:G\to G/L$ is a \emph{fiber bundle} if and only if the point $Le$ (and hence all other points of $G/L$) has a local cross
section (\cite{Steenrod1}). The characterization of the pair $(G,L)$ for which $\pi$ is a fiber bundle is an unsolved problem.
But there are many partial answers to this problem. Just for example we refer a few results: If $G$ is locally compact
with finite covering dimension then for any closed subgroup $L$, $\pi$ is a fiber bundle (\cite{Nagami1}). This is the case
foe a (finite dimensional) Lie group $G$. Also if $L$ is a compact Lie group and $G$ is an arbitrary topological group
then $\pi$ is a fiber bundle (\cite{Gleason1}).
\begin{theorem}\label{T2}
Let $G$ be a strongly amenable topological group and $L$ be a closed subgroup of $G$ for which the canonical map $\pi:G\to X$ is a fiber
bundle, where $X=G/L$ is the right $G$-space of right cosets. Let $E$ be a Banach space and suppose that there are given
continuous maps $h:G\to E^*$ and $f:X\to E^*$ such that
$$\|f(xy)-f(x)-h(y)\|\leq\delta\hspace{10mm}(x\in X, y\in G).$$
Then there exist continuous maps $H:G\to E^*$ and $F:X\to E^*$ such that $H(yz)=H(y)+H(z)$, $F(xy)=F(x)+H(y)$, $\|H(y)-h(y)\|\leq\delta$, and
$\|F(x)-f(x)\|\leq2\delta$.
\end{theorem}
\begin{proof}
It follows from Lemmas \ref{L1} and \ref{L2} that there is a bounded linear map $m$ from $\b{CB}(X;E^*)$ to $E^*$ satisfying properties analogous to those of $\tilde{m}$ in Lemma \ref{L2}. For every $y\in G$ the map
$x\mapsto f(xy)-f(x)$ is continuous, and bounded by $\|h(y)\|+\delta$. So we may define a map $H:G\to E^*$ by
$H(y)=m(x\mapsto f(xy)-f(x))$.  Analogous to the proof of Theorem \ref{T1},
\begin{align*}
H(y)+H(z)&=m(x\mapsto f(xy)-f(x))+m(x\mapsto f(xz)-f(x))\\
&=m(x\mapsto f(xy)-f(x))+m(x\mapsto f((xy)z)-f(xy))\\
&=m(x\mapsto f(x(yz))-f(x))=H(yz).
\end{align*}
The norm inequality is proved by,
\begin{align*}
\|H(y)-h(y)\|&=\|m(x\mapsto f(xy)-f(x))-h(y)\|\\
&=\|m(x\mapsto f(xy)-f(x)-h(y))\|\\
&\leq\|m\|\|x\mapsto f(xy)-f(x)-h(y)\|_\infty\leq\delta.
\end{align*}
Now it follows from Lemma \ref{L3} that $H$ is also continuous.

The map $z\mapsto f(xz)-H(z)$ is continuous, and bounded by $\|f(x)\|+2\delta$.
A result analogous to Lemma \ref{L2} shows that there is a bounded linear map $n$ from $\b{CB}(G;E^*)$ to $E^*$ satisfying
properties analogous to those of $\tilde{m}$ in Lemma \ref{L2}, for the left $G$-space $G$.
So we can define a map $F:X\to E^*$ by $F(x)=n(z\mapsto f(xz)-H(z))$. Analogous to the proof of Theorem \ref{T1},
\begin{align*}
F(x)+H(y)&=n(z\mapsto f(xz)-H(z))+ H(y)\\
&=n(z\mapsto f(xz)-H(z)+H(y))\\
&=n(z\mapsto f(xz)-H(y^{-1}z))\\
&=n(z\mapsto f(xyz)-H(z))=F(xy).
\end{align*}
The norm inequality is proved by,
\begin{align*}
\|F(x)-f(x)\|&=n(z\mapsto f(xz)-H(z))-f(x)\\
&=n(z\mapsto f(xz)-H(z)-f(x))\leq2\delta.
\end{align*}
It remains to show that $F$ is continuous. This can be done as follows.
Let $x$ be an arbitrary element of $X$ and let $(x_\lambda)_\lambda$ be a net in $X$ converging to $x$. Suppose that $U$
is an open subset of $X$ containing $x$ and $s:U\to G$ be a local cross section around $x$. There exists a $\lambda_0$ such that
$x_\lambda\in U$ for every $\lambda\geq\lambda_0$. Let $s(x)=y$ and $s(x_\lambda)=y_\lambda$ for $\lambda\geq\lambda_0$.
So $y_\lambda\to y$ in $G$. We have $x_\lambda=Ly_\lambda=Lyy^{-1}y_\lambda=x(y^{-1}y_\lambda)$. So $F(x_\lambda)=F(x)+H(y^{-1}y_\lambda)$ and,
$$\lim_{\lambda}F(x_\lambda)=F(x)+\lim_\lambda H(y^{-1}y_\lambda)=F(x)+H(e)=F(x).$$
This completes the proof.
\end{proof}
Analogue of the above theorem is satisfied for more general groups and $G$-spaces but with stronger conditions on $f$:
\begin{theorem}\label{T3}
Let $G$ be an amenable topological group, $X$ be a right $G$-space and $E$ be a Banach space. Suppose that there are given
maps $h:G\to E^*$ and $f:X\to E^*$ such that $f\in\b{UC}(X;E^*)$ and,
$$\|f(xy)-f(x)-h(y)\|\leq\delta\hspace{10mm}(x\in X, y\in G).$$
Then there are maps $H:G\to E^*$ and $F:X\to E^*$ such that $H$ is continuous and $F\in\b{UC}(X;E^*)$, and such that
$H(yz)=H(y)+H(z)$, $F(xy)=F(x)+H(y)$, $\|H(y)-h(y)\|\leq\delta$, and $\|F(x)-f(x)\|\leq2\delta$.
\end{theorem}
\begin{proof}
Let $y,z\in G$ be fixed. Since $f$ is right uniform, for every $\epsilon>0$ there is an open $U_\epsilon\subseteq G$ containing $zy$ such that
$\|f(xzy)-f(xy')\|<\epsilon$ for every $x\in X$ and every $y'\in U_\epsilon$. Since $G$ is a topological group there is an open $U'\subseteq G$
containing $z$ such that $z'y\in U_\epsilon$ for every $z'\in U'$. It follows that $\|f(xzy)-f(xz'y)\|<\epsilon$ for every $z'\in U'$ and $x\in X$.
This shows that the map $x\mapsto f(xy)$ is right uniform. From this and the assumptions it follows that the map $x\mapsto f(xy)-f(x)$ belongs
to $\b{RUCB}(X;E^*)$.

Since $f$ is right uniform, for every $y\in G$ and $\epsilon>0$ there is an open $U_{y,\epsilon}\subseteq G$ containing $y$ such that
$\|f(xy)-f(x)-f(xy')+f(x)\|=\|f(xy)-f(xy')\|<\epsilon$ for every $x\in X$ and every $y'\in U_{y,\epsilon}$.
It follows that the map $y\mapsto(x\mapsto f(xy)-f(x))$ is continuous from $G$ to $\b{RUCB}(X,E^*)$. On the other hand
it follows from Lemmas \ref{L1} and \ref{L2} that there is a linear map $m$ from $\b{RUCB}(X;E^*)$ to $E^*$
satisfying properties analogous to those of $\tilde{m}$ in Lemma \ref{L2}.
So we can define a continuous map $H:G\to E^*$ by $H(y)=m(x\mapsto f(xy)-f(x))$
which, analogous to Theorem \ref{T2}, satisfies $H(yz)=H(y)+H(z)$ and $\|H(y)-h(y)\|\leq\delta$.

Now let $x\in X$ and $y\in G$ be fixed. Since $f$ is left uniform, for every $\epsilon>0$ there is an open $V_\epsilon\subseteq X$ containing $xy$
such that $\|f(xyz)-f(x'z)\|<\epsilon$ for every $z\in G$ and every $x'\in U_\epsilon$. Since the action is continuous there is an open
$U_{\epsilon}\subseteq G$ containing $y$ such that $xy'\in V_\epsilon$ for every $y'\in U_\epsilon$. Thus $\|f(xyz)-f(xy'z)\|<\epsilon$ for
every $y'\in U_\epsilon$ and every $z\in G$. It follows that the map $y\mapsto f(xy)$ is in $\b{LU}(G,E^*)$. Since $H$ is a continuous homomorphism,
$H$ is also a uniform map. So for every fixed $x\in X$ the map $y\mapsto f(xy)-H(y)$ is in $\b{LU}(G,E^*)$. On the other hand this map
is continuous, and bounded by $\|f(x)\|+2\delta$. Thus we have shown that this map belongs to $\b{LUCB}(G,E^*)$.

Since $f$ is left uniform, for every $x$ and every $\epsilon>0$ there is an open
$V_{x,\epsilon}\subseteq X$ containing $x$ such that $\|f(xy)-H(y)-f(x'y)+H(y)\|=\|f(xy)-f(x'y)\|<\epsilon$ for every $x'\in V_{x,\epsilon}$
and every $y\in G$. It follows that the map $x\mapsto(y\mapsto f(xy)-H(y))$ is continuous from $X$ to $\b{LUCB}(G,E^*)$. On the other hand
a result analogous to Lemma \ref{L2} shows that there is a linear map $n$ from $\b{LUCB}(G;E^*)$ to $E^*$ satisfying
properties analogous to those of $\tilde{m}$ in Lemma \ref{L2}, for the left $G$-space $G$.
So we can define a continuous map $F:X\to E^*$ by $F(x)=n(z\mapsto f(xz)-H(z))$ and, analogous to Theorem \ref{T2}, it is proved that
$F(xy)=F(x)+H(y)$ and $\|F(x)-f(x)\|\leq2\delta$ for every $x\in X$ and $y\in G$.

It remains to show that $F$ is a uniform map. Let $x\in X$. Since $F$ is continuous, for every $\epsilon>0$ there is an open $V\subseteq X$
containing $x$ such that $\|F(x)-F(x')\|<\epsilon$ for every $x'\in V$, and hence $\|F(xy)-F(x'y)\|=\|F(x)+H(y)-F(x')-H(y)\|<\epsilon$
for every $y\in G$. This shows that $F$ is left uniform. Since $H$ is continuous, for $y\in G$ and $\epsilon>0$ there is an open
$U\subseteq G$ containing $y$ such that for every $y\in U$, $\|H(y)-H(y')\|<\epsilon$, and hence
$\|F(xy)-F(xy')\|=\|F(x)+H(y)-F(x)-H(y')\|<\epsilon$ for every $x\in X$. This shows that $F$ is right uniform.
\end{proof}

\end{document}